\documentclass[11pt]{article}

\usepackage{psfrag,graphics,color,epsfig,subfigure}
\usepackage{amsmath,amssymb,amscd, mathrsfs,tikz}
\usepackage{enumerate,afterpage}
\usepackage{pifont}
\allowdisplaybreaks
\usetikzlibrary{shapes,arrows}

\allowdisplaybreaks

\newcommand{\ts}[1]		{{\textstyle{#1}}}
\newcommand{\demi}		{{\ts{\frac{1}{2}}}}
\newcommand{\op}[1]		{{\mathcal{{#1}}}}

\newcommand{\done}{\ding{182}}
\newcommand{\dtwo}{\ding{183}}
\newcommand{\dthree}{\ding{184}}

\newcommand{\opBmp}[1]	{{\op{B}_{k}}}

\newcommand{\N}{\mathbb{N}}
\newcommand{\M}{\mathbb{M}}
\newcommand{\Z}{\mathbb{Z}}
\newcommand{\R}{\mathbb{R}}
\newcommand{\ds}{\displaystyle}

\newcommand{\be}{\begin{equation}}
\newcommand{\ee}{\end{equation}}
\newcommand{\bea}{\begin{eqnarray}}
\newcommand{\eea}{\end{eqnarray}}
\newcommand{\ba}{\begin{array}}
\newcommand{\ea}{\end{array}}

\newcommand{\er}[1]{\hbox{(\ref{#1})}}

\newcommand{\nn}			{{\nonumber}}

\newtheorem{theorem}            {Theorem}[section]
\newtheorem{corollary}          [theorem]{Corollary}

\newtheorem{definition}         [theorem]{Definition}
\newtheorem{lemma}              [theorem]{Lemma}
\newtheorem{sideremark}         [theorem]{Remark}

\newtheorem{sideeg}           [theorem]{Example}
\newtheorem{sideconj}           [theorem]{Conjecture}
\newtheorem{sideassumption}   [theorem]{Assumption}

\newenvironment{remark}         {\begin{sideremark}\rm}{\end{sideremark}}

\newenvironment{assumption} {\begin{sideassumption}\it}{\end{sideassumption}}
\newenvironment{proof}		{{\it Proof:}}{\hfill{$\blacksquare$}}

\title{Max-plus fundamental solution semigroups for a class of difference Riccati equations}
\author{Huan Zhang \qquad\qquad Peter M. Dower%
\thanks{This paper was not presented at any IFAC 
meeting. This research is supported by grants FA2386-12-1-4084 and DP120101549 from AFOSR and the Australian Research Council. 
Zhang and Dower are with the Department of Electrical and Electronic Engineering, University of Melbourne, Melbourne, Victoria 3010, Australia.
Email: \{hzhang5,pdower\}@unimelb.edu.au}
}

\begin{document}           


\maketitle

\begin{abstract}

Recently, a max-plus dual space fundamental solution semigroup for a class of difference Riccati equation (DRE) has been developed. This fundamental solution semigroup is represented in terms of the kernel of a specific max-plus linear operator that plays the role of the dynamic programming evolution operator in a max-plus dual space. In order to fully understand connections between this dual space fundamental solution semigroup and evolution of the value function of the underlying optimal control problem, a new max-plus primal space fundamental solution semigroup for the same class of difference Riccati equations is presented. Connections and commutation results between this new primal space fundamental solution semigroup and the recently developed dual space fundamental solution semigroup are established. 

{\bf Keywords.} Difference Riccati equations, Max-plus algebra, Fundamental solution semigroup.    

\end{abstract}

\section{Introduction}

The difference Riccati equation (DRE) is of fundamental importance in the study and solution of optimal control and filtering problems formulated in discrete time \cite{AM:89}, \cite{B:05}. In the control context, a solution of the DRE (or of the corresponding differential Riccati equation in continuous time) characterises the controller that solves the associated optimal control problem, and the optimal cost associated with that solution. One of the important topics in the investigation of both difference and differential Riccati equations is the characterisation and representation of all solutions via some form of {\em fundamental solution} \cite{DM:73}, \cite{KTKR:06}. For example, in the (continuous time) differential Riccati equation case, the well-known Davison-Maki fundamental solution \cite{DM:73} exploits the solution of the corresponding Hamiltonian differential equation via a Bernoulli substitution technique.  Alternatively, a max-plus fundamental solution developed  in \cite{M:08} exploits the linearity of the dynamic programming evolution operator associated with the attendant optimal control problem, with respect to the max-plus algebra. It has been demonstrated that continuous and discrete time formulations of this max-plus fundamental solution facilitate efficient solution of the differential and difference Riccati equations respectively \cite{DM1:11},  \cite{DM2:12}, \cite{DM1:12}, \cite{M:08}, \cite{ZD3:13}. In the latter case, this max-plus fundamental solution has also been recently applied in investigating existence of solutions and finite escape properties of DRE solutions \cite{ZD4:13}. 

In the development of this max-plus fundamental solution for either continuous \cite{M:08} or discrete time \cite{ZD3:13}, a specific duality pairing is employed that uniquely identifies the value function on a given horizon, which resides in a {\em primal space}, with a corresponding element of a {\em max-plus dual space}, via the Legendre-Fenchel transform. This dual space element is used to define the kernel of a max-plus linear max-plus integral operator indexed with the same time horizon. By virtue of the aforementioned duality pairing, max-plus linearity of the dynamic programming evolution operator, and the semigroup property enjoyed by this dynamic programming evolution operator, it is shown that the set of all such time horizon indexed max-plus linear max-plus integral operators defines a semigroup in the dual space. In particular, the value function corresponding to any terminal payoff can be evolved to longer time horizons in the dual space by application of elements of this semigroup of max-plus linear max-plus integral operators. As evolution of the value function is equivalent to evolution of the difference or differential Riccati equation solution from an initial condition specified by the Hessian of the terminal payoff, and this terminal payoff may be selected arbitrarily from a large set of semiconvex functions, the aforementioned semigroup may be regarded as a {\em max-plus dual space fundamental solution semigroup} for the corresponding difference or differential Riccati equation.

In view of the existing max-plus dual space fundamental solution developed for difference Riccati equations \cite{ZD3:13}, the aim of this paper is to explore the existence, and subsequent properties, of a max-plus primal space fundamental solution semigroup for the same DRE. In principle, the construction of such a primal space fundamental solution semigroup involves the representation of the dynamic programming evolution operator for each time horizon in terms of a specific max-plus linear max-plus integral operator indexed by the same time horizon. The resulting class of time horizon indexed operators takes the same form as in the dual space case, but with each element  defined entirely on the primal space. It is shown that the primal and dual space fundamental solution semigroups are in fact isomorphic.

In terms of organisation, Section \ref{sec:LQR-DRE} introduces the class of DREs and associated discrete-time linear quadratic optimal control problem of interest. Section \ref{sec:space} defines the max-plus primal and dual spaces. Section \ref{sec:mp-dual} summarises the existing max-plus dual space fundamental solution semigroup \cite{ZD3:13}, followed by an analogous development of the new max-plus primal space fundamental solution semigroup in Section \ref{sec:mp-primal}. Section \ref{sec:connection} includes a detailed analysis of the connection between the max-plus primal and dual space fundamental solution semigroups.

In terms of notation, $\R, \N,\Z_{\ge0}$ denote the sets of reals, natural numbers, and non-negative integers respectively. Two sets of extended reals are denoted by $\R^{-}\doteq\R\cup\{-\infty\}$ and $\R^+\doteq \R\cup\{\infty\}$. The set of $n\times n$ real, symmetric matrices is denoted by $\M^{n\times n}\doteq\{P\in\R^{n\times n}| P=P^T\}$. Given $P\in\mathbb{M}^{n\times n}$, $P>0$ (respectively $P\ge 0$) denotes positive (nonnegative)
definiteness of $P$. The triple $(\R^{-},\oplus, \otimes)$ denotes a semiring, representing the max-plus algebra, with addition and multiplication operations defined respectively by $a\oplus b\doteq\max\{a,b\}$ and $a\otimes b\doteq a+b$. The max-plus integral of a function $f:\R^n\rightarrow\R^{-}$ is defined as $\int_{\R^n}^\oplus f(x)\,dx\doteq \sup_{x\in\R^n}f(x)$.

\section{The difference Riccati equation and optimal control}
\label{sec:LQR-DRE}

Attention is restricted to DREs of the form
\begin{align}
\label{eq:DRE}
P_{k+1}=\op{R}(P_k),\quad P_0\in\M^{n\times n},
\end{align}  
with the Riccati operator $\op{R}:\M^{n\times n}\rightarrow\M^{n\times n}$ defined by
\begin{align}
\label{eq:Riccati-op}
\op{R}(P)\doteq \Phi+A^TPA+A^TPB(\gamma^2\,I-B^TPB)^{-1}B^TPA.
\end{align}
Here, $A\in\R^{n\times n}$, $B\in\R^{n\times m}, n\ge m$, $\Phi\in\M^{n\times n}, \Phi>0$ are real matrices, and $\gamma\in\R$, $\gamma>0$, is fixed. DRE \er{eq:DRE} is an example of the indefinite difference Riccati equation \cite{SS:98}, \cite{ZDG:96}  as the so-called Popov matrix 
$
\Pi=\left[\ba{cc} \Phi  &0\\0&-\gamma^2\,I\ea\right]
$
is indefinite. For $k\in\N$, define $\op{R}_k$ iteratively by 
\begin{align}
\label{eq:op-Rk}
\op{R}_{k+1}=\op{R}\circ\op{R}_k,
\end{align}
so that  $
\op{R}_k=\underbrace{\op{R}\circ \op{R}\circ \cdots \circ \op{R}}_{k\text{ times}}. 
$
A matrix $P_k\in\M^{n\times n}$ is a solution of DRE \er{eq:DRE} at step $k\in\N$ corresponding to initial condition $P_0$ if $\gamma^2\,I-B^T\op{R}_i(P_0)B>0$ for $i=0,1,\cdots,k-1$. This solution can be expressed as $P_k=\op{R}_k(P_0)$. DRE \er{eq:DRE} arises in the study of linear quadratic regulator (LQR) \cite{AM:89}, where the underlying linear system dynamics are given by
\begin{align}
	x_{k+1} & = A\, x_k+B\, w_k\,, \quad x_0=x.
	\label{eq:system}
\end{align} 
The value function $W_k:\R^n\rightarrow\R$, $k\in\Z_{\ge0}$ of interest is defined by
\begin{align}
	W_k(x)
	& \doteq
	\sup_{w_{0,k-1}\in (\R^{m})^k}  J_k(x;\, w_{0,k-1})\,
	\label{eq:value}
\end{align}
via the total payoff
\begin{align}
\nn
	&J_k(x;w_{0,k-1})\doteq\sum_{i=0}^{k-1}\left({\frac{1}{2}}x_i^T\, \Phi \, x_i - {\frac{\gamma^2}{2}}\, |w_i|^2 \right) + {\frac{1}{2}}x_k^TP_0 x_k\,.
\end{align}
Here, $w_{0,k-1}\in(\R^{m})^k$ denotes an input sequence for system \er{eq:system} on interval $[0,k-1]$. It is well known \cite{AM:89} that $W_k$ is a quadratic function of form $W_k(x)=\frac{1}{2}x^T\,P_k\,x, x\in\R^n$, in which $P_k\in\M^{n\times n}$ is the solution at time $k$ of the DRE \er{eq:DRE} with initial condition $P_0$.

Define the (one-step) dynamic programming evolution operator $\op{S}$ by
\begin{align}
	& \hspace{-2mm}
	(\op{S} \, \phi)(x)
	\!\doteq\!\!
	\sup_{w\in\R^m} \{ \ts{\frac{1}{2} x^T \Phi\, x-{\frac{\gamma^2}{2}} \, |w|^2+ \phi(A x + B w)} \},
	\label{eq:op-DPP-1}
\end{align}
and $k$-step  dynamic programming evolution operator $\op{S}_k, k\in\N$, iteratively by
\begin{align}
	\op{S}_{k+1} \, \phi
	& \doteq\op{S} \left( \op{S}_k\, \phi \right)\,.
	\label{eq:op-DPP-k}
\end{align}
As a matter of convention, define $\op{S}_0\doteq\mathcal{I}$ to be the identity operator. Dynamic programming implies that the set $\{ \mathcal{S}_k, k\in\Z_{\ge 0}\}$ defines a semigroup of operators, see \cite{ZD3:13}, and an element $\op{S}_k$ propagates the terminal payoff $W_0:\R^n\rightarrow \R^-$ defined by
\begin{align}
\label{eq:W0}
W_0(x)=\frac{1}{2}x^T P_0 x, \quad x\in\R^n,
\end{align} 
to the value function via $W_k=\op{S}_kW_0$.  A fundamental property of the operator $\op{S}_k$ is that it is linear over the max-plus algebra, that is, $\op{S}_k(a\otimes \phi_1\oplus\phi_2)=a\otimes \op{S}_k\phi_1\oplus\op{S}_k\phi_2$, see \cite{M:06}, \cite{ZD3:13}.

\section{Max-plus primal and dual space}
\label{sec:space}
In order to introduce the existing max-plus dual space fundamental solution semigroup \cite{ZD3:13}, and subsequently develop a new max-plus primal space fundamental solution semigroup, it is necessary to first introduce the max-plus primal and dual spaces, and the associated duality pairing.
\begin{definition}
\label{def:semi-conv} (\cite{M:06}, \cite{M:08}) Given ${K}\in\M^{n\times n}$, a function $\phi:\R^n\rightarrow\R^-$ is uniformly semiconvex with respect to ${K}$ if the function $x\mapsto\phi(x)+\demi x^T {K} x:\R^n\rightarrow\R^-$ is convex on $\R^n$. Analogously, a function $\phi:\R^n\rightarrow\R^-$ is uniformly semiconcave with respect to ${K}$ if the function $x\mapsto\phi(x)-\demi x^T {K} x:\R^n\rightarrow\R^-$ is concave  on $\R^n$.
\end{definition}
The spaces of uniformly semiconvex and semiconcave functions are defined with respect to ${K}\in\M^{n\times n}$ by
\begin{align}
\label{eq:space-semiconvex}
& \hspace{-3mm}
\op{S}^{K}_+(\R^n)\doteq\left\{\!\phi:\R^n\rightarrow\R^-\biggl | \ba{c}\phi\text{ is uniformly}\\\text{semiconvex w.r.t. } {K} \ea\! \right\}\!,
\\
& \hspace{-3mm}
\op{S}^{K}_-(\R^n)\doteq\left\{\!\phi:\R^n\rightarrow\R^-\biggl |\ba{c}\phi\text{ is uniformly}\\\text{semiconcave w.r.t. } {K} \ea\!\right\}\!.
\label{eq:space-semiconcave}
\end{align}

As per \cite{M:06}, \cite{ZD3:13}, define a pair of operators $\op{D}_\psi$ and $\op{D}_\psi^{-1}$ for a given $M\in\M^{n\times n}$ by \begin{align}
	\op{D}_\psi \, \phi
	& = \left( \op{D}_\psi \, \phi \right)(\cdot)
	\doteq -\int_{\R^n}^\oplus \psi(x,\cdot) \otimes \left( - \phi(x) \right)\, dx\,,
	\label{eq:op-dual}
	\\
	\op{D}_\psi^{-1} \, \hat\phi
	& =  \left( \op{D}_\psi^{-1} \, \hat\phi \right)(\cdot)
	\doteq \int_{\R^n}^\oplus \psi(\cdot,z) \otimes \hat\phi(z) \, dz\,,
	\label{eq:op-inv-dual}
\end{align}
where the function $\psi:\R^n\times\R^n\rightarrow\R$ is a quadratic function defined for all $x,z\in\R^n$ by
\begin{align}
\label{eq:mp-basis-quad}
\psi(x,z)\doteq\frac{1}{2} (x-z)^T \, M\, (x-z).
\end{align}
Assume the following restrictions on $M\in\M^{n\times n}$ throughout.
\begin{assumption}
\label{ass:M}
Given $\gamma\in\R_{>0}$ and $B\in\R^{n\times m}$ as per \er{eq:Riccati-op} and \er{eq:system}, the matrix $M\in\M^{n\times n}$ in \er{eq:mp-basis-quad} satisfies the inequalities
\begin{align}
\label{ineq:M2}
\gamma^2\,I-B^T\op{R}_k(M)B&>0, \quad \forall~k\in\Z_{\ge0},
\\
\label{ineq:M}
\op{R}(M)-M&>0,
\\
\label{ineq:M3}
MB(\gamma^2\,I-B^TMB)^{-1}B^TM&>0.
\end{align}
\end{assumption}
Inequality \er{ineq:M2} requires that a particular solution $P_k=\op{R}_k(M)$ of DRE \er{eq:DRE} with initial condition $P_0=M$ exists for all $k\in\N$.  Inequality \er{ineq:M} implies (see Theorem \ref{thm:value-space}) that this particular solution $\op{R}_k(M)$ is nondecreasing in $k\in\Z_{\ge0}$ and satisfies $\op{R}_k(M)>M$ for all $k\in\N$. Inequality \er{ineq:M3} is useful in deriving the max-plus primal space fundamental solution semigroup in Section \ref{sec:mp-primal}.

The following result shows that the operators $\op{D}_\psi$ of \er{eq:op-dual} and $\op{D}_\psi^{-1}$ of \er{eq:op-inv-dual}  can be used to define a duality between the spaces $\op{S}_+^{-M}(\R^n)$ and $\op{S}_-^{-M}(\R^n)$ of \er{eq:space-semiconvex} and \er{eq:space-semiconcave}, where $M\in\M^{n\times n}$ is as per \er{eq:mp-basis-quad}.
\begin{theorem}
\label{thm:dual}
The operator $\op{D}_\psi$ of \er{eq:op-dual} is a bijection from $\op{S}_+^{-M}(\R^n)$ to $\op{S}_-^{-M}(\R^n)$ with inverse operator $\op{D}_\psi^{-1}$ given by \er{eq:op-inv-dual}.
\end{theorem}
\begin{proof}
It is first shown that 
\begin{align}
\label{prof:thm-dual-a}
\phi\in\op{S}_+^{-M}(\R^n)&\Longrightarrow \op{D}_\psi\phi\in\op{S}_-^{-M}(\R^n),
\\
\hat\phi\in\op{S}_-^{-M}(\R^n)&\Longrightarrow \op{D}_\psi^{-1}\hat\phi\in\op{S}_+^{-M}(\R^n).
\label{prof:thm-dual-b}
\end{align}
For any $\phi\in\op{S}^{-M}_+(\R^n)$, by Definition \ref{def:semi-conv} of uniform semiconvexity with respect to $-M$, the function $\phi_
+:\R^n\rightarrow\R^-$ defined by $\phi_+(x)\doteq\phi(x)-\ts{\frac{1}{2}}x^TMx$ is convex on $\R^n$. By convex duality (e.g., \cite{BV:04}), the convex conjugate $\phi_+^\ast:\R^n\rightarrow\R^+$ defined by
\begin{align}
\label{eq:prof-dual-space-1}
\phi_+^\ast(\eta)=\int_{\R^n}^\oplus \eta^Tx\otimes(-\phi_+(x))\,dx
\end{align}
is convex on $\R^n$. Hence, by \er{eq:op-dual},
\begin{align}
\nn
&(\op{D}_\psi\phi)(z)=-\int_{\R^n}^\oplus\psi(x,z)\otimes (-\phi(x))\,dx
\\\nn
                              &\quad=-\max_{x\in\R^n}\left\{-z^TMx+\ts{\frac{1}{2}}x^TMx-\phi(x)\right\}-\ts{\frac{1}{2}}z^TMz
                             \\\nn
                              &\quad=-\max_{x\in\R^n}\left\{-z^TMx-\phi_+(x)\right\}-\ts{\frac{1}{2}}z^TMz
                             \\\nn
                             &\quad=-\phi_+^\ast(-Mz)-\ts{\frac{1}{2}}z^TMz.
\end{align}
 Thus, the function $\tilde\phi:\R^n\rightarrow\R^-$ defined by
$\tilde\phi(z)\doteq(\op{D}_\psi\phi)(z)+\ts{\frac{1}{2}}z^TMz= -\phi_+^\ast(-Mz), z\in\R^n$, is concave from \er{eq:prof-dual-space-1}. That is, $\op{D}_\psi\phi\in\op{S}^{-M}_-(\R^n)$ by Definition \ref{def:semi-conv}.

To show that \er{prof:thm-dual-b} holds, fix any $\hat\phi\in\op{S}^{-M}_-(\R^n)$.  By Definition \ref{def:semi-conv} of uniform semiconcavity with respect to $-M$, the function $\hat\phi_
-:\R^n\rightarrow\R^+$ defined by $\hat\phi_-(z)\doteq-(\hat\phi(z)+\ts{\frac{1}{2}}z^TMz), z\in\R^n$, is convex on $\R^n$. By convex duality, the convex conjugate $\hat\phi_-^\ast:\R^n\rightarrow\R^-$ defined by
\begin{align}
\label{eq:prof-dual-space-2}
{\hat\phi_-}^\ast(\xi)\doteq\int_{\R^n}^\oplus\xi^Tz-\hat\phi_-(z)\,dz  
\end{align}
is a convex function on $\R^n$. Hence, by \er{eq:op-inv-dual},
\begin{align}
\nn
&(\op{D}_\psi^{-1}\hat\phi)(x)=\int_{\R^n}^\oplus \psi(x,z)\otimes\hat\phi(z)\,dz
\\\nn
                                            &=\max_{z\in\R^n}\left\{ -x^TMz+\ts{\frac{1}{2}}z^TMz+\hat\phi(z)\right\}+\ts{\frac{1}{2}}x^TMx
                                            \\\nn
                                            &=\max_{z\in\R^n}\left\{ -x^TMz-\hat\phi_-(z)\right\}+\ts{\frac{1}{2}}x^TMx
                                            \\\nn
                                            &={\hat\phi}_-^\ast(-Mx)+\ts{\frac{1}{2}}x^TMx.
\end{align}
 Thus, the function $\bar\phi:\R^n\rightarrow\R^-$ defined by $\bar\phi(x)\doteq(\op{D}_\psi^{-1}\hat\phi)(x)-\ts{\frac{1}{2}}x^TMx={\hat\phi}_-^\ast(-Mx)$, $x\in\R^n$, is convex. That is, $\op{D}_\psi^{-1}\hat\phi\in\op{S}^{-M}_+(\R^n)$ by Definition \ref{def:semi-conv}.

The assertion that the operator $\op{D}_\psi^{-1}$ is the inverse of $\op{D}_\psi$ is proved in \cite[Theorem 2.9]{M:06}. Thus, for any $\hat\phi\in\op{S}_-^{-M}(\R^n)$, there exists an element $\phi\doteq\op{D}_\psi^{-1}\hat\phi\in\op{S}_{+}^{-M}(\R^n)$ such that $\op{D}_\psi\phi=\op{D}_\psi(\op{D}_\psi^{-1}\hat\phi)=\hat\phi$. This, together with \er{prof:thm-dual-a}, proves that $\op{D}_\psi$ is bijection from $\op{S}_+^{-M}(\R^n)$ to $\op{S}_-^{-M}(\R^n)$.
\end{proof}

For the purpose of studying solutions of the DRE \er{eq:DRE}, the domain of operator $\op{D}_\psi$  can be restricted to a space of quadratic functions specified by
\begin{align}
\label{def:Q+}
& \hspace{-3mm}
\mathcal{Q}_+^{-M}(\R^n) \doteq\left\{\phi:\R^n\rightarrow\R\biggl | \ba{c}\phi(x)=\ts{\frac{1}{2}}x^T\Omega x,\\\Omega\in\M^{n\times n}, \Omega> M\ea\right\}\!.
\end{align}
Given any $\phi\in\mathcal{Q}_+^{-M}$, the function $\check\phi:\R^n\rightarrow\R$ defined by $\check\phi(x)\doteq\phi(x)+\textstyle{\frac{1}{2}}x^T(-M)x=\ts{\frac{1}{2}}x^T(\Omega-M)x, x\in\R^n$, is convex on $\R^n$. This shows that $\mathcal{Q}_+^{-M}(\R^n)\subset\op{S}_+^{-M}(\R^n)$. Define the range of operator $\op{D}_\psi$ over the space $\mathcal{Q}_+^{-M}(\R^n)$ by
$
\mathsf{ran}(\op{D}_\psi)\doteq\left\{ \op{D}_\psi \phi\,| \phi\in\mathcal{Q}_+^{-M}(\R^n)\right\}.
$
In order to explicitly characterise $\mathsf{ran}(\op{D}_\psi)$, define a matrix operation $\Upsilon:\M^{n\times n}\rightarrow\M^{n\times n}$ by
\begin{align}
\label{def:Upsilon}
\Upsilon(\Omega)\doteq M(M-\Omega)^{-1}M-M
\end{align}
for $\Omega\in\M^{n\times n}$ such that $\Omega>M$. It can be verified directly that the inverse of $\Upsilon$ is 
\begin{align}
\label{def:Upsilon-inv}
\Upsilon^{-1}(\Omega)&\doteq -M(M+\Omega)^{-1}M+M=-\Upsilon(-\Omega)
\end{align}
for all $\Omega\in\M^{n\times n}$ such that $\Omega<-M$. 
Define 
\begin{align}
\label{def:Q-}
& \hspace{-3mm}
\mathcal{Q}_-^{-M}(\R^n)\doteq\left\{\phi:\R^n\rightarrow\R\biggl |\ba{c} \phi(x)=\ts{\frac{1}{2}}x^T\Upsilon(\Omega)x \\ \Omega\in\M^{n\times n}, \Omega>M\ea\right\}\!.
\end{align}

\begin{theorem}
\label{thm:dual-value}
The set $\mathcal{Q}_-^{-M}(\R^n)$ is the range of the operator $\op{D}_\psi$ over the space $\mathcal{Q}_+^{-M}(\R^n)$. That is, $\mathcal{Q}_-^{-M}(\R^n)=\mathsf{ran}(\op{D}_\psi)$.
\end{theorem}
\begin{proof}
For any $\phi\in\mathcal{Q}_+^{-M}(\R^n)$, let $\Omega\in\M^{n\times n}, \Omega>M$, be such that $\phi(x)=\frac{1}{2}x^T\Omega x$ for all $x\in\R^n$. From \er{eq:op-dual}, \er{def:Upsilon},
\begin{align}
\nn
& (\op{D}_\psi\phi)(z) = -\int_{\R^n}^\oplus \psi(x,z) \otimes \left( - \phi(x) \right)\, dx
\\\nn
          &=-\max_{x\in\R^n}\{\ts{\frac{1}{2}}(x-z)^T \! M(x-z)-\ts{\frac{1}{2}}x^T\Omega x\}
          =\ts{\frac{1}{2}}z^T\Upsilon(\Omega)z.
\end{align}
Thus, each element in $\mathsf{ran}(\op{D}_\psi)$ corresponds to an element in $\mathcal{Q}_-^{-M}(\R^n)$. That is, $\mathcal{Q}_-^{-M}(\R^n)=\mathsf{ran}(\op{D}_\psi)$.
\end{proof}

The following result is an immediate consequence of Theorems \ref{thm:dual} and \ref{thm:dual-value}. 
\begin{corollary}
\label{cor:duality}
The operator $\op{D}_\psi$ of \er{eq:op-dual} is a bijection from $\mathcal{Q}_+^{-M}(\R^n)$ to $\mathcal{Q}_-^{-M}(\R^n)$ with inverse $\op{D}_\psi^{-1}$ given by \er{eq:op-inv-dual}. 
\end{corollary}
In view of definitions \er{eq:op-dual} and \er{eq:op-inv-dual}, Theorem \ref{thm:dual-value}, and Corollary \ref{cor:duality}, $\mathcal{Q}^{-M}_+(\R^n)$ is referred to as a {\em max-plus primal space} and $\mathcal{Q}_-^{-M}(\R^n)$ is referred to as a {\em max-plus dual space}. The dynamic programming evolution operator $\op{S}_k$ of \er{eq:op-DPP-k} propagates the value function $W_k$ of \er{eq:value} in the  
max-plus primal space $\mathcal{Q}_+^{-M}(\R^n)$. The domain of $\op{S}_k, k\in\N$, is defined by
\begin{align}
\label{def:dom-Sk}
& \hspace{-2mm}
\mathsf{dom}(\op{S}_k)\!\doteq\!\left\{\!\ba{c}
							\phi\in\mathcal{Q}_+^{-M}(\R^n) \\
							\phi(x)=\ts{\frac{1}{2}}x^T\Omega x
					\ea  \bigg | \ba{cc} 
\Omega\in\M^{n\times n} \text{ such}\\ \text{that }\op{R}_k(\Omega)~\text{exists}
 \ea \!
 \right\}\!.
\end{align}

In order to show that the value function $W_k$ stays in $\mathcal{Q}_+^{-M}(\R^n)$ for any horizon $k\in\N$ and any terminal payoff $W_0\in\mathcal{Q}_+^{-M}(\R^n)$, the following monotonicity property of the Riccati operator $\op{R}_k$ from \cite{ZD4:13} is useful. 
\begin{lemma}
\label{lem:mono-Rk}
Suppose that solutions $P_k^1\doteq\op{R}_k(P_0^1), P_k^2\doteq\op{R}_k(P_0^2)$  of DRE \er{eq:DRE} exist at time $k\in\Z_{\ge0}$ corresponding to initial conditions  $P_0^1,P_0^2\in\M^{n\times n}$ . Then,
\begin{align}
\label{eq:lem-mon-Rk}
P_0^1\le P_0^2\Longrightarrow P_k^1=\op{R}_k(P_0^1)\le\op{R}_k(P_0^2)=P_k^2.
\end{align} 
\end{lemma}
\if{false}
\begin{proof}
Let $W_k^i, i=1,2$ denote the value function \er{eq:value} at time $k\in\Z_{\ge0}$ corresponding to initial conditions $W_0^i(x)=\frac{1}{2} x^T P_0^i x$, $x\in\R^n$, where $P_0^i$, $i=1,2$ are as per the Lemma statement. Then, by definition  \er{eq:op-Rk} of $\op{R}_k$, $W_k^i(x)=\frac{1}{2} x^T \op{R}_k(P_0^i) x,x\in\R^n$. Fix any $x\in\R^n$. Applying definition \er{eq:value} of the value function $W_k^i$ yields
 \begin{align}
 \nn
 &\ts{\frac{1}{2}}x^T P_k^1x=\ts{\frac{1}{2}}x^T\op{R}_k(P_0^1)x=W_k^1(x)
 \\\nn
 &\quad\quad=\sup_{w_{0,k-1}\in(\R^m)^k}\left\{\ba{c}\ts{\ds{\sum_{i=0}^{k-1}}}\left( {\frac{1}{2}}x_i^T\Phi x_i-\frac{1}{2}\gamma^2w_i^Tw_i\right)\\+{\frac{1}{2}}x_k^T P_0^1x_k\ea\right\}
 \\\nn
 &\quad\quad\le\sup_{w_{0,k-1}\in(\R^m)^k}\left\{\ba{c}\ts{\ds{\sum_{i=0}^{k-1}}}\left( {\frac{1}{2}}x_i^T\Phi x_i-\frac{1}{2}\gamma^2w_i^Tw_i\right)\\+{\frac{1}{2}}x_k^T P_0^2x_k\ea\right\}
 \\\nn
 &\quad\quad=W_k^2(x)=\ts{\frac{1}{2}}x^T\op{R}_k(P_0^2)x=\ts{\frac{1}{2}}x^TP_k^2x.
 \end{align} 
 Hence, as $x\in\R^n$ is arbitrary, the right-hand side inequality of \er{eq:lem-mon-Rk} follows. 
 \end{proof}

\fi
\begin{theorem}
\label{thm:value-space}
Suppose that Assumption \ref{ass:M} holds. Fix any $k\in\Z_{\ge0}$, and any initial value function $W_0\in\mathsf{dom}(\op{S}_k)$ of the form \er{eq:W0}, with $P_0>M$. Then, the value function $W_k=\op{S}_kW_0$ satisfies $W_k\in\mathcal{Q}_+^{-M}(\R^n)$. 
\end{theorem}
\begin{proof}
From the definition \er{def:dom-Sk} of $\mathsf{dom}(\op{S}_k)$, the solution $\op{R}_k(P_0)$ exists at time $k$ and the value function $W_k(x)=\frac{1}{2}x^T\op{R}_k(P_0)x, x\in\R^n$. Since $P_0>M$, it follows that $\op{R}_k(P_0)\ge\op{R}_k(M)$ from the monotonicity of the operator $\op{R}_k$ from Lemma \ref{lem:mono-Rk}. Applying the Riccati operator $\op{R}$ to both sides of the inequality $\op{R}(M)>M$ yields $\op{R}_2(M)=\op{R}(\op{R}(M))\ge\op{R}(M)>M$ by Assumption \ref{ass:M} and Lemma \ref{lem:mono-Rk}. Repeating the process yields $\op{R}_k(M)>M$ . Hence, $\op{R}_k(P_0)\ge\op{R}_k(M)>M$. Thus, the value function $W_k(x)=\frac{1}{2}x^T\op{R}_k(P_0)x, x\in\R^n,$ belongs to $\mathcal{Q}_+^{-M}(\R^n)$ according to \er{def:Q+}.
\end{proof}

\begin{remark}
\label{rmk:value-space}
Theorem \ref{thm:value-space} implies that $W_k\in\mathcal{Q}_+^{-M}(\R^n)$ at horizon $k\in\Z_{\ge0}$ if $W_0\in\mathsf{dom}(\op{S}_k)\subset\mathcal{Q}_+^{-M}(\R^n)$. Consequently, Theorem \ref{thm:dual} implies that the max-plus dual $\widehat{W}_k\doteq\op{D}_\psi W_k$ exists. Theorem \ref{thm:dual-value} subsequently implies that $\widehat{W}_k\in\mathcal{Q}_-^{-M}(\R^n)$ with a representation $\widehat{W}_k(z)=\frac{1}{2}z^T\Upsilon(\op{R}_k(P_0))z$, for all $z\in\R^n$, where the operator $\Upsilon:\R^{n\times n}\rightarrow \R^{n\times n}$ is as defined in \er{def:Upsilon}.Ä
\end{remark}

\section{Max-plus dual space fundamental solution semigroup}
\label{sec:mp-dual}
Inspired by the continuous time analysis of \cite{M:08}, and the infinite dimensional analysis of \cite{DM1:11},  \cite{DM2:12}, \cite{DM1:12}, a new max-plus dual space fundamental solution semigroup for the corresponding discrete time DRE \er{eq:DRE} has recently been developed, see \cite{ZD3:13}. With the aim of providing context for the new max-plus primal space fundamental solution semigroup presented in this paper, the development in \cite{ZD3:13} is summarised below.

Define an auxiliary value function $\widehat{\mathrm{S}}_k:\R^n\times\R^n\rightarrow\R$ by applying the operator  $\op{S}_k$ to the functions $\psi(\cdot,z)$
\begin{align}
\label{eq:hatSk}
\widehat{\mathrm{S}}_k(x,z)&\doteq \left(\op{S}_k\psi(\cdot,z)\right)(x).
\end{align}
From \cite[Theorem 3.1]{ZD3:13}, $\widehat{\mathrm{S}}_k$ is a quadratic of the form 
\begin{align}
\label{eq:quadr-hatSk}
\widehat{\mathrm{S}}_k(x,z)&=\frac{1}{2}\left[\ba{c} x\\z \ea\right]^T Q_k \left[\ba{c} x\\z\ea\right],
\end{align}
in which the Hessian $Q_k$ may be expressed in block form by $Q_k=\left[\ba{cc}Q_{k}^{11}&Q_{k}^{12}\\ Q_{k}^{21}& Q_{k}^{22}\ea\right]\in\M^{2n\times 2n}$. Theorem 3.1 in \cite{ZD3:13} shows that the matrices $Q_k$ may be generated iteratively for all $k\in\Z_{\ge0}$ by\begin{align}
\label{eq:dynamics-Q}
Q_{k+1}^{11}&=\op{R}(Q_k^{11}),
\\\nn
Q_{k+1}^{12}&=A^TQ_{k}^{12}+A^TQ_{k}^{11}B(\gamma^2I-B^TQ_{k}^{11}B)^{-1}B^TQ_{k}^{12},
\\\nn
Q_{k+1}^{21}&=(Q_{k+1}^{12})^T,
\\\nn
Q_{k+1}^{22}&=Q_{k}^{22}+Q_{k}^{21}B(\gamma^2I-B^TQ_{k}^{11}B)^{-1}B^TQ_{k}^{12},
\end{align}
with initial condition
$
Q_0=\left[\ba{cc} M &-M\\-M&M\ea\right].
$
From the first equation in \er{eq:dynamics-Q}, each element of $\{Q_k^{11}, k\in\Z_{\ge0}$\} satisfies DRE \er{eq:DRE} with $Q_0^{11}=M$. That is,  $Q_k^{11}=\op{R}_k(M)$ for any $k\in\Z_{\ge0}$.  Inequality \er{ineq:M2} in Assumption \ref{ass:M} implies that $Q_k^{11}$ exists for all $k\in\N$.  Hence $Q_k$ exists for all $k\in\N$ by \er{eq:dynamics-Q}. As shown in Theorem \ref{thm:value-space}, it follows that $Q_k^{11}>M$ for $k\in\N$, from inequality \er{ineq:M} in Assumption \ref{ass:M}.  This implies that the operator $\op{D}_\psi$ of \er{eq:op-dual} can be applied to $\widehat{\mathrm{S}}_k(\cdot, z)$ of \er{eq:hatSk}, for each $z\in\R^n$, to yield a function $\mathrm{B}_k(\cdot, z):\R^n\rightarrow\R$ defined by
\begin{align}
\label{eq:Bk}
\mathrm{B}_k(y,z)&\doteq (\op{D}_\psi\widehat{\mathrm{S}}_k(\cdot,z))(y)
\\\nn
&=-\int_{\R^n}^\oplus\psi(x,y)\otimes(-\widehat{\mathrm{S}}_{k}(x,z))\,dx
\end{align}
for all $y\in\R^n$. In \cite{M:08}, \cite{ZD3:13}, the function $\mathrm{B}_k:\R^n\times\R^n\rightarrow\R$ is used to define a max-plus linear max-plus integral operator  $\op{B}_k$ by
\begin{align}
\label{eq:op-Bk}
(\op{B}_k\,a)(y)\doteq\int_{\R^n}^\oplus \mathrm{B}_k(y,z)\otimes a(z)\,dz
\end{align}
for all $y\in\R^n$. It may be shown \cite{M:08} that $\op{B}_k$ and $\op{S}_k$ of \er{eq:op-DPP-k} are related by
\begin{align}
\label{eq:op-Bk-simi}
\op{B}_k=\op{D}_\psi\,\op{S}_k\op{D}_\psi^{-1}
\end{align}
for all $k\in\N$. From Corollary \ref{cor:duality} and Theorem \ref{thm:value-space},  $\op{D}_\psi$ is a bijection from $\mathcal{Q}_+^{-M}(\R^n)$ to $\mathcal{Q}_-^{-M}(\R^n)$ with inverse $\op{D}_\psi^{-1}$, and $\op{S}_k$ is a map from $\mathcal{Q}_+^{-M}{(\R^n)}$ to $\mathcal{Q}_+^{-M}{(\R^n)}$. Consequently, \er{eq:op-Bk-simi} implies that $\op{B}_k$ is a map from $\mathcal{Q}_-^{-M}{(\R^n)}$ to $\mathcal{Q}_-^{-M}{(\R^n)}$. Since $\{\op{S}_k, k\in\N\}$ defines a semigroup \cite{ZD3:13}, it follows that $\{\op{B}_k, k\in\Z_{\ge0}\}$ also defines a semigroup by inspection of \er{eq:op-Bk-simi}. Let $\widehat{W}_k=\op{D}_\psi W_k, k\in\Z_{\ge0}$, be the max-plus dual of the value function $W_k$. Applying \er{eq:op-Bk-simi},
\begin{align}
\label{eq:prop-dual}
\op{B}_k\widehat{W}_0&=\op{B}_k(\op{D}_\psi W_0)=(\op{D}_\psi\,\op{S}_k\op{D}_\psi^{-1})(\op{D}_\psi W_0)
\\\nn
                                     &=\op{D}_\psi(\op{S}_k W_0)=\op{D}_\psi W_k=\widehat{W}_k.
\end{align}
That is, $\{\op{B}_k, k\in\Z_{\ge0}\}$ defines a semigroup of max-plus linear max-plus integral operators that propagate the max-plus dual of the value function with respect to the time horizon $k\in\Z_{\ge0}$. This provides an alternative way of propagating any terminal payoff $W_0\in\mathcal{Q}_+^{-M}(\R^n)$ to its corresponding value function $W_k$ at time $k\in\Z_{\ge0}$. This is summarised via the commutation diagram of Figure \ref{fig:propagations} and the following steps:
\begin{center}
\vspace{2mm}
\parbox[c]{8cm}{\centering
\begin{itemize}
\item[\done]
Map the terminal payoff $W_0$ into the dual space $\mathcal{Q}_-^{-M}(\R^n)$ by $\widehat{W}_0=\op{D}_\psi W_0$.

\item[\dtwo]
Propagate $\widehat{W}_0$ via $\op{B}_k$ to $\widehat{W}_k=\op{B}_k\widehat{W}_0$.

\item[\dthree]
Recover the value function $W_k=\op{D}_\psi^{-1}\widehat{W}_k$ via the inverse dual operator $\op{D}_\psi^{-1}$ of \er{eq:op-inv-dual}.

\end{itemize}
\vspace{2mm}
}
\end{center}
\begin{figure}[h]
\begin{center}
\begin{tikzpicture}[node distance=2cm, auto]
  \node (C) {$W_0$};
  \node(D)[right of=C,node distance=3cm]{$W_k$};
  \node (P) [below of=C,node distance=2cm] {$\widehat{W}_0$};
  \node (Ai) [right of=P,node distance=3cm] {$\widehat{W}_k$};
  \draw[->] (C) to node [swap] {$\op{D}_\psi$} (P);
  \draw[->] (C) to node []{$\op{S}_k$} (D);
  \draw[->] (P) to node [] {$\op{B}_k$} (Ai);
   \draw[->](Ai) to node []{$\op{D}_\psi^{-1}$}(D);
\end{tikzpicture}
\end{center}
\caption{\footnotesize  Commutation diagram for propagation of  $W_k$ by $\op{S}_k$ of \er{eq:op-DPP-k} or by $\op{B}_{k}$ of \er{eq:op-Bk}.}
\label{fig:propagations}
\end{figure}

By inspection of \er{eq:op-Bk} and \er{eq:op-Bk-simi}, the set of kernels $\{\mathrm{B}_k, k\in\N\}$ defined via \er{eq:Bk} also define a semigroup. In this case, the associate binary operation used for propagation to longer time horizons is the max-plus convolution
\begin{align}
\label{eq:propag-Bk}
\mathrm{B}_{k_1+k_2}(y,z)=\int_{\R^n}^\oplus\mathrm{B}_{k_1}(y,\rho)\otimes \mathrm{B}_{k_2}(\rho,z)\,d\rho
\end{align}
for any $k_1,k_2\in\N$ and $y,z\in\R^n$. Furthermore, it has been shown \cite{ZD3:13} that $\mathrm{B}_k$ is a quadratic of the form
\begin{align}
\label{eq:Thetak}
\mathrm{B}_k(y,z)=\frac{1}{2}\left[\ba{c}y\\ z\ea\right]^T\Theta_k \left[\ba{c}y\\ z\ea\right],
\end{align}
in which the Hessian $\Theta_k$ may be expressed in block form by $\Theta_k= \left[\ba{cc}\Theta_{k}^{11}&\Theta_{k}^{12}\\\Theta_{k}^{21}&\Theta_{k}^{22}\ea\right] \in\M^{2n\times 2n}$. By inspection of \er{eq:propag-Bk} and \er{eq:Thetak}, the set of Hessians $\{\Theta_k, k\in\N\}$ also defines a semigroup, with a corresponding associative binary operation specified by a matrix operation $\circledast$ defined by
\begin{align}
\label{eq:propaga-Theta}
&\Theta_{k_1+k_2}=\Theta_{k_1}\circledast\Theta_{k_2}
\\\nn
&\doteq\left[\ba{cc} \Theta_{k_1}^{11}-\Theta_{k_1}^{12}\Pi_{(k_1,k_2)}^{-1}\Theta_{k_1}^{21} &-\Theta_{k_1}^{12}\Pi_{(k_1,k_2)}^{-1}\Theta_{k_2}^{12}\\-\Theta_{k_2}^{21}\Pi_{(k_1,k_2)}^{-1}\Theta_{k_1}^{21}&\Theta_{k_2}^{22}-\Theta_{k_2}^{21}\Pi_{(k_1,k_2)}^{-1}\Theta_{k_2}^{12} \ea\right],
\end{align}
for all $k_1,k_2\in\N$, in which $\Pi_{(k_1,k_2)}\doteq\Theta_{k_1}^{22}+\Theta_{k_2}^{11}$ (see \cite{ZD3:13}). Conditions that guarantee that $\Pi_{(k_1,k_2)}>0$ for all $k_1,k_2\in\N$ are given in Theorem 4.2 of \cite{ZD3:13}. Hessian propagation via \er{eq:propaga-Theta} is initialised with $\Theta_1=\left[\ba{cc}\Theta_1^{11}&\Theta_1^{12}\\\Theta_1^{21}&\Theta_1^{22}\ea\right]$, where an explicit calculation of $\mathrm{B}_1$ via \er{eq:Bk}, and an application of \er{eq:Thetak}, yields
\begin{align}
\label{eq:Theta1}
\Theta_1^{11}&=M(M-Q_1^{11})^{-1}M-M,
\\\nn
\Theta_1^{12}&=M(M-Q_1^{11})^{-1}Q_1^{12},
\\\nn
\Theta_1^{21}&=Q_1^{21}(M-Q_1^{11})^{-1}M,
\\\nn
\Theta_1^{22}&=Q_1^{21}(M-Q_1^{11})^{-1}Q_1^{12}+Q_1^{22}.
\end{align}
Here, $Q_1=\left[\ba{cc}Q_1^{11}&Q_1^{12}\\Q_1^{21}&Q_1^{22}\ea\right]$ is as per \er{eq:dynamics-Q}, with
\begin{align}
\nn
Q_1^{11}&= \Phi+A^TMA+A^TMB(\gamma^2\,I-B^TMB)^{-1}B^TMA,
\\\nn
Q_1^{12}&=-A^TM-A^TMB(\gamma^2\,I-B^TMB)^{-1}B^TM,
\\\label{eq:Q1}
Q_1^{21}&= -MA-MB(\gamma^2\,I-B^TMB)^{-1}B^TMA,
\\\nn
Q_1^{22}&=M+MB(\gamma^2\,I-B^TMB)^{-1}B^TM.
\end{align}
The semigroup $\{\Theta_k, k\in\N\}$ defined via \er{eq:propaga-Theta} is referred to here as the max-plus dual space fundamental solution semigroup for the DRE \er{eq:DRE}.

From Remark \ref{rmk:value-space}, the max-plus dual of the value function $\widehat{W}_k$ is a quadratic $\widehat{W}_k(z)=\frac{1}{2}z^T\Upsilon(P_k) z, z\in\R^n$, where $P_k=\op{R}_k(P_0)$. Denote  the Hessian of 
$\widehat{W}_k$ by $
O_k\doteq\Upsilon(P_k).
$
Applying \er{eq:op-Bk} and \er{eq:prop-dual} yields
 \begin{align}
 \nn
& \widehat{W}_k(z)=\ts{\frac{1}{2}}z^T O_k z=\int_{\R^n}^\oplus \mathrm{B}_k(z, \rho)\otimes \widehat{W}_0(\rho) \, d\rho
\\\nn
                            &=\int_{\R^n}^\oplus  \frac{1}{2}\left[\ba{c}z\\ \rho\ea\right]^T \left[\ba{cc}\Theta_{k}^{11}&\Theta_{k}^{12}\\\Theta_{k}^{21}&\Theta_{k}^{22}\ea\right]  \left[\ba{c}z\\ \rho\ea\right]  \otimes \frac{1}{2} \rho^T O_0 \rho\,d\rho
\\\nn
         &=\frac{1}{2} z^T (\Theta_k^{11}-\Theta_k^{12}(O_0+\Theta_k^{22})^{-1}\Theta_k^{21}) z
\end{align}
for all $z\in\R^n$ and $k\in\N$. Thus, the Hessian $O_k$ of the max-plus dual $\widehat{W}_k$ can be computed for any $k\in \N$ via the max-plus dual space fundamental solution semigroup $\{\Theta_k, k\in\N\}$  by 
\begin{align}
\label{eq:prop-fund}
O_k=\Psi^{d}_k(O_0),
\end{align}
where the operation $\Psi_k^d:\M^{n\times n}\rightarrow \M^{n\times n}$ for each $k\in\N$ is defined by
\begin{align}
\label{eq:math-F}
\Psi^d_k(\Omega)\doteq\Theta_k^{11}-\Theta_k^{12}(\Omega+\Theta_k^{22})^{-1}\Theta_k^{21}.
\end{align} 
Note that $O_0+\Theta_k^{22}<0$ is necessary for the representation \er{eq:prop-fund}, see \cite{ZD4:13}.  From Remark \ref{rmk:value-space}, $O_k$ and $P_k$ are related by the operation $\Upsilon$ of \er{def:Upsilon} and $\Upsilon^{-1}$ of \er{def:Upsilon-inv} via $O_k=\Upsilon(P_k)$ and $P_k=\Upsilon^{-1}(O_k)$, respectively. The representation of solution  $P_k=\op{R}_k(P_0)$ of DRE \er{eq:DRE} via the max-plus dual space fundamental solution semigroup $\{\Theta_k, k\in\N\}$ is then given by 
\begin{align}
\label{eq:dual-fund}
O_0=\Upsilon(P_0),
\,\,
O_k=\Psi_k^{d}(O_0),
\,\,
P_k=\Upsilon^{-1}(O_k)
\end{align}
for all $k\in\N$. The max-plus dual space fundamental solution semigroup $\{\Theta_k, k\in\N\}$ can be computed using the propagation rule \er{eq:propaga-Theta} initialised with $\Theta_1$ given in \er{eq:Theta1}. After $\{\Theta_k, k\in\N\}$ is computed, a solution $P_k=\op{R}_k(P_0)$ of DRE \er{eq:DRE} corresponding to any allowable initial condition $P_0>M$  can be obtained directly using formula \er{eq:dual-fund}. Note that $\{\Theta_k, k\in\N\}$ only needs to be computed once, and its computation is independent of the initial condition $P_0$. 

\section{Max-plus primal space fundamental solution semigroup}
\label{sec:mp-primal}

Equation \er{eq:dual-fund} provides a representation of solutions  $P_k=\op{R}_k(P_0)$ of DRE \er{eq:DRE} via the max-plus dual space fundamental solution semigroup $\{\Theta_k, k\in\N\}$. In this section, a new max-plus fundamental solution semigroup is developed that allows a simpler representation of the solution $P_k=\op{R}_k(P_0)$.

From \er{eq:dynamics-Q}, matrices $Q_k^{22}, k\in\Z_{\ge0}$, satisfy the iteration
\begin{align}
\nn
Q_{k+1}^{22}=Q_{k}^{22}+Q_{k}^{21}B(\gamma^2I-B^TQ_{k}^{11}B)^{-1}B^TQ_{k}^{12}
\end{align}
with initial condition $Q_0^{22}=M$. Since $Q_k^{11}=\op{R}_k(M), k\in\Z_{\ge0}$, it follows that $\gamma^2\,I-B^TQ^{11}_kB>0$ by inequality \er{ineq:M} for all $k\in\Z_{\ge0}$. Thus, $Q_{k+1}^{22}\ge Q_k^{22}$, for all $k\in\Z_{\ge0}$. Inequality \er{ineq:M2} implies that 
\begin{align}
\nn
Q_1^{22}&=Q_{0}^{22}+Q_{0}^{21}B(\gamma^2I-B^TQ_{0}^{11}B)^{-1}B^TQ_{0}^{12}
\\\label{eq:mono-Q22}
&=M+MB(\gamma^2\,I-B^TMB^T)^{-1}B^TM
\\\nn
& >M.
\end{align}
Here, $Q_0^{21}=Q^{12}_0=-M$ is used, as per \er{eq:dynamics-Q}. Thus, $Q_k^{22}>M$ for all $k\in\N$. Consequently, the operator $\op{D}_\psi$ of \er{eq:op-dual} can be applied to the function $\widehat{\mathrm{S}}_k(x,\cdot)$ of \er{eq:hatSk}, for each $x\in\R^n$, to yield a new function $\mathrm{S}_k(x,\cdot):\R^n\rightarrow\R$
\begin{align}
\label{eq:kernel-primal}
& \hspace{-2mm} \mathrm{S}_{k}(x,y)\doteq\left(\op{D}_\psi\widehat{\mathrm{S}}_{k}(x,\cdot)\right)(y)
\\\nn
&
=-\int_{\R^n}^\oplus\psi(z,y)\otimes(-\widehat{\mathrm{S}}_{k}(x,z))\,dz
\end{align}
for all $x,y\in\R^n$ and $k\in\N$.
\if{false}

\begin{remark}
\label{rmk:mp-Green}
In \cite{dm3:13}, the concept of max-plus Green's function was proposed. It can be shown that the function $\mathrm{S}_k$ of \er{eq:kernel-primal} can be interpreted as a max-plus Green's function, with
\begin{align}
\mathrm{S}_k(x,y)=(\op{S}_k\delta_y)(x)
\end{align}
for all $x,y\in\R^n$, where the max-plus delta function $\delta_y:\R^n\rightarrow\R^-$ centred at $y\in\R^n$ is defined by 
$
\delta_y(x)\doteq\left\{\ba{cc} 0,&x=y, \\-\infty, & x\neq y,\ea\right.
$
for all $x\in\R^n$.
\end{remark}

\fi
Applying the inverse dual operator $\op{D}^{-1}_\psi$ of \er{eq:op-inv-dual} to $\mathrm{S}_k(x,\cdot)$ yields  a representation of $\widehat{\mathrm{S}}_{k}$ in terms of $\mathrm{S}_{k}$, with
\begin{align}
\label{eq:kernel-primal-inv}
\widehat{\mathrm{S}}_{k}(x,z)&=\left(\op{D}^{-1}_\psi\mathrm{S}_{k}(x,\cdot)\right)(z)
= \! \int_{\R^n}^\oplus \!\! \psi(z,y)\otimes\mathrm{S}_{k}(x,y)\,dy.
\end{align}

Define a max-plus integral operator by
\begin{align}
\label{eq:op-tildeS}
(\widetilde{\op{S}}_k\phi)(x)\doteq\int_{\R^n}^\oplus \mathrm{S}_{k}(x,y)\otimes\, \phi(y)\,dy
\end{align}
for all $x\in\R^n, k\in\N$, and $\phi\in\mathcal{Q}_+^{-M}(\R^n)$ such that the max-plus integral in \er{eq:op-tildeS} is finite. The next theorem shows that $\widetilde{\op{S}}_k$ coincides with the dynamic programming evolution operator $\op{S}_k$ of \er{eq:op-DPP-k} on $\mathsf{dom}(\op{S}_k)$, see \er{def:dom-Sk}.
\begin{theorem}
\label{thm:kernel-DPP-op}
Suppose that  $\mathsf{dom}(\op{S}_k)\neq\emptyset$ at horizon $k\in\N$. Then, for any $\phi\in\mathsf{dom}(\op{S}_k)$,
\begin{align}
\label{eq:thm:kernel-DPP-op}
\op{S}_k\phi=\widetilde{\op{S}}_k\phi.
\end{align}
\end{theorem}
\begin{proof}
It is shown first that 
\begin{align}
\label{eq:kernel-hatS}
\left(\widehat{\op{S}}_{k}\hat\phi\right)(x)=\int_{\R^n}^\oplus \widehat{\mathrm{S}}_k(x,z)\otimes \,\hat\phi(z)\,dz
\end{align}
for any $x\in\R^n, k\in\N$, and $\hat\phi\in\mathcal{Q}_-^{-M}(\R^n)$ such that $\op{D}_\psi^{-1}\hat\phi\in\mathsf{dom}(\op{S}_k)$. By max-plus linearity of the operator $\op{S}_k$,
\begin{align}
	& (\widehat{\op{S}}_{k}\hat\phi)(x) =(\op{S}_k\op{D}^{-1}_\psi\,\hat\phi)(x)
	\nn\\
	& =\left(\op{S}_k\int_{\R^n}^\oplus \psi(\cdot,z)\otimes \,\hat\phi(z)\,dz \right)(x)
\nn\\
                                          &=\int_{\R^n}^\oplus \left(\op{S}_k\psi(\cdot,z)\right)(x)\otimes \,\hat\phi(z)\,dz
                                          \nn\\
                                          &
					=\int_{\R^n}^\oplus \widehat{\mathrm{S}}_k(x,z)\otimes \,\hat\phi(z)\,dz.
\nn
\end{align}
Then, for any $x\in\R^n$, $\phi\in\mathsf{dom}(\op{S}_k)\subset\mathcal{Q}_+^{-M}(\R^n)$,   \er{eq:kernel-primal}, \er{eq:kernel-primal-inv}, and \er{eq:kernel-hatS} imply that
\begin{align}
\nn
&\left(\op{S}_k\phi\right)(x) =(\op{S}_k\op{D}_\psi^{-1})(\op{D}_\psi\phi)(x)
\\\nn
&=\int_{\R^n}^\oplus \widehat{\mathrm{S}}_{k}(x,z)\otimes (\op{D}_\psi\phi)(z)\,dz
\\\nn
                             &=\int_{\R^n}^\oplus\left(\int_{\R^n}^\oplus \mathrm{S}_{k}(x,y)\otimes \psi(z,y)\,dy \right)\otimes (\op{D}_\psi\phi)(z)\,dz
\\\nn
                            &=\int_{\R^n}^\oplus\mathrm{S}_{k}(x,y)\otimes\left(\int_{\R^n}^\oplus\psi(y,z)\otimes(\op{D}_\psi\phi)(z)\,dz\right)\,dy
\\\nn
                            &=\int_{\R^n}^\oplus\mathrm{S}_{k}(x,y)\otimes\phi(y)\,dy
                           =(\widetilde{\op{S}}_k\phi)(x),
\end{align}
where the third equality uses the fact that $\psi(x,z)=\psi(z,x)$ for all $x,z\in\R^n$.
\end{proof}

Theorem \ref{thm:kernel-DPP-op} and \er{eq:op-tildeS} show that the dynamic programming evolution operator $\op{S}_k$ is a max-plus linear max-plus integral operator with kernel $\mathrm{S}_k$ defined by \er{eq:kernel-primal}. 
\begin{corollary}
For any $\phi\in\mathsf{dom}(\op{S}_k)\subset\mathcal{Q}^{-M}_+(\R^n), k\in\N$, the dynamic programming evolution operator 
 $\op{S}_k$ of \er{eq:op-DPP-k} satisfies
\begin{align}
\label{eq:Sk-kernel-represent}
({\op{S}}_k\phi)(x)=\int_{\R^n}^\oplus \mathrm{S}_{k}(x,y)\otimes\, \phi(y)\,dy
\end{align}
for all $x\in\R^n$.
\end{corollary}
Using \er{eq:Sk-kernel-represent}, the value function $W_k=\op{S}_k W_0$ defined with respect to $W_0\in\mathsf{dom}(\op{S}_k)$ can be expressed by
\begin{align}
\label{eq:Wk-primal-kernel}
W_k(x)=\int_{\R^n}^\oplus \mathrm{S}_k(x,y)\otimes W_0(y)\,dy
\end{align}
for all $x\in\R^n$ and $k\in\N$. Analogous to $\mathrm{B}_k$ of \er{eq:Bk}, the function $\mathrm{S}_k, k\in\N$, is a quadratic function of the form
\begin{align}
\label{eq:quad-primal}
\mathrm{S}_{k}(x,y)&=\frac{1}{2}\left[\ba{c}x \\ y \ea\right]^T\Lambda_{k} \left[\ba{c}x \\ y \ea\right]
\end{align}
for all $x,y\in\R^n$, in which the Hessian $\Lambda_k$ is defined in block form by $\Lambda_k= \left[\ba{cc}\Lambda_{k}^{11}&\Lambda_{k}^{12}\\\Lambda_{k}^{21}&\Lambda_{k}^{22}\ea\right]\in\M^{2n\times 2n}$. 
Analogous to the propagation rules for $\mathrm{B}_k$ and its Hessian $\Theta_k$ specified in \er{eq:propag-Bk} and \er{eq:propaga-Theta},  $\mathrm{S}_k$ and its Hessian $\Lambda_k$ follow similar propagation rules. 
\begin{theorem}
\label{thm:primal-propa}
For any $k_1,k_2\in\N$, the functions  $\mathrm{S}_{k_1}, \mathrm{S}_{k_2}$ defined by \er{eq:kernel-primal} satisfy
\begin{align}
\label{eq:convolution-kernel}
\mathrm{S}_{k_1+k_2}(x,y)=\int_{\R^n}^\oplus \mathrm{S}_{k_1}(x,\rho)\otimes\mathrm{S}_{k_2}(\rho,y)\,d\rho
\end{align}
for all $x,y\in\R^n$ and the matrices $\Lambda_{k_1}, \Lambda_{k_2}$ of \er{eq:quad-primal} satisfy
\begin{align}
\label{eq:circledast-Lamk}
\Lambda_{k_1+k_2}=\Lambda_{k_1}\circledast\Lambda_{k_2},
\end{align}
where the $\circledast$ operation is defined as per \er{eq:propaga-Theta}.
\end{theorem}
\begin{proof}
For any $x\in\R^n$, $k_1,k_2\in\N$ and $\phi\in\mathsf{dom}(\op{S}_{k_1+k_2})$, from \er{eq:thm:kernel-DPP-op},
\begin{align}
\nn
&(\op{S}_{k_1+k_2}\phi)(x)=\int_{\R^n}^\oplus \mathrm{S}_{k_1+k_2}(x,y)\otimes\phi(y)\,dy
\\\nn
               &=(\op{S}_{k_1}(\op{S}_{k_2}\phi))(x)
              =\int_{\R^n}^\oplus\mathrm{S}_{k_1}(x,\rho)\otimes(\op{S}_{k_2}\phi)(\rho)\,d\rho
\\\nn
              &=\int_{\R^n}^\oplus\mathrm{S}_{k_1}(x,\rho)\otimes\left(\int_{\R^n}^\oplus\mathrm{S}_{k_2}(\rho,y)\otimes\phi(y)\,dy\right)\,d\rho
             \\\nn
             &=\int_{\R^n}^\oplus\left(\int_{\R^n}^\oplus \mathrm{S}_{k_1}(x,\rho)\otimes\mathrm{S}_{k_2}(\rho,y)\,d\rho\right)\otimes\phi(y)\,dy.
\end{align}
Since $\phi\in\mathsf{dom}(\op{S}_{k_1+k_2})$ is arbitrary, \er{eq:convolution-kernel} follows. Applying the quadratic form \er{eq:quad-primal}  of  $\mathrm{S}_{k_1+k_2}, \mathrm{S}_{k_1}$, and $\mathrm{S}_{k_2}$, and evaluating the quadratic maximisation with respect to $\rho\in\R^n$ in \er{eq:convolution-kernel} explicitly yields \er{eq:circledast-Lamk}.
\end{proof}

Similar to $\Theta_1$ in \er{eq:Theta1}, the initial condition $\Lambda_1=\left[\ba{cc}\Lambda_1^{11}&\Lambda_1^{12}\\\Lambda_1^{21}&\Lambda_1^{22}\ea\right]$ can be obtained from definition \er{eq:kernel-primal}, with
\begin{align}
\nn
\Lambda_1^{11}&=Q_1^{12}(M-Q_1^{22})^{-1}Q_1^{21}+Q_1^{11},
\\\nn
\Lambda_1^{12}&=Q_1^{12}(M-Q_1^{22})^{-1}M,
\\\label{eq:Lambda1}
\Lambda_1^{21}&=M(M-Q_1^{22})^{-1}Q_1^{21},
\\\nn
\Lambda_1^{22}&=M(M-Q_1^{22})^{-1}M-M,
\end{align}
where $Q_1$ is as per \er{eq:Q1}.

Recalling that the set  $\{\op{S}_k, k\in\N\}$, along with operator composition, defines a semigroup of operators, the sets $\{\mathrm{S}_k,k\in\N\}$ and $\{\Lambda_k, k\in\N\}$ define semigroups, with respective associative binary operations defined by the max-plus convolution  \er{eq:convolution-kernel} and the $\circledast$ operation \er{eq:circledast-Lamk}. The semigroup $\{\Lambda_k, k\in\N\}$ is referred to here as the max-plus primal space fundamental solution for the DRE \er{eq:DRE}. Using representation \er{eq:Wk-primal-kernel}, the semigroup $\{\Lambda_k,k\in\N\}$ can be used to derive a new representation of the solution $P_k=\op{R}_k(P_0)$ of the DRE \er{eq:DRE}. 

\begin{theorem}
\label{thm:primal-mp-fund}
Given the semigroup $\{\Lambda_k, k\in\N\}$ of \er{eq:quad-primal} and \er{eq:circledast-Lamk},  for any $P_0\in\M^{n\times n}$ such that $\Lambda^{22}_k+P_0<0$, the solution   $P_k=\op{R}_k(P_0)$ at horizon $k\in\N$ of the DRE \er{eq:DRE}  is given by
\begin{equation}
P_k=\Psi^p_k(P_0),
\label{eq:Pk-primal-repre}
\end{equation} 
where $\Psi^p_k:\M^{n\times n}\rightarrow\M^{n\times n}$ is defined by
\begin{align}
\label{def:Psi-p}
\Psi^p_k(\Omega)\doteq\Lambda_k^{11}-\Lambda_k^{12}(\Omega+\Lambda_k^{22})^{-1}\Lambda_k^{21},
\end{align}
for any $\Omega\in\M^{n\times n}$ satisfying $\Omega+\Lambda_k^{22}<0$.
\end{theorem}
\begin{proof}
Fix $k\in\N$ and an arbitrary $P_0\in\M^{n\times n}$ such that $P_0+\Lambda_k^{22}<0$.  The solution $P_k=\op{R}_k(P_0)$ of the DRE \er{eq:DRE} is the Hessian of the value function $W_k$ of \er{eq:value} with terminal payoff $W_0(x)=\frac{1}{2}x^TP_0 x, x\in\R^n$. Applying \er{eq:Wk-primal-kernel} yields for any $x\in\R^n$,
\begin{align}
\nn
&\ts{\frac{1}{2}}x^T P_k x=W_k(x)=(\op{S}_k W_0)(x)
\\\nn
            &=\int_{\R^n}^\oplus \mathrm{S}_k(x,y)\otimes \ts{\frac{1}{2}}y^TP_0 y\,dy
\\\nn
            &=\max_{y\in\R^n}\left\{{\frac{1}{2}}\left[\ba{c}x \\ y \ea\right]^T  \left[\ba{cc}\Lambda_{k}^{11}&\Lambda_{k}^{12}\\\Lambda_{k}^{21}&\Lambda_{k}^{22}\ea\right]   \left[\ba{c}x \\ y \ea\right]+{\frac{1}{2}}y^TP_0 y\right\}
\\\nn
            &=\ts{\frac{1}{2}}x^T(\Lambda_k^{11}-\Lambda_k^{12}(P_0+\Lambda_k^{22})^{-1}\Lambda_k^{21}) x.
\end{align}
Since this holds for all $x\in\R^n$, \er{eq:Pk-primal-repre} follows. 
\end{proof}

\begin{remark}
It has been shown \cite{ZD4:13} that the condition $\Lambda_k^{22}+P_0<0$ is a necessary and sufficient condition for the existence of the solution $P_k=\op{R}_k(P_0)$ of DRE \er{eq:DRE} at time $k\in\N$.
\end{remark}

The commutation diagram for computation of $P_k=\op{R}_k(P_0)$ via the max-plus primal and dual fundamental solution semigroup is shown in Figure \ref{fig:dual-primal}.
\begin{figure}[h]
\begin{center}
\begin{tikzpicture}[node distance=2cm, auto]
  \node (C) {$P_0$};
  \node(D)[right of=C,node distance=3cm]{$P_k$};
  \node (P) [below of=C,node distance=2cm] {$O_0$};
  \node (Ai) [right of=P,node distance=3cm] {$O_k$};
  \draw[->] (C) to node [swap] {$\Upsilon$} (P);
  \draw[->] (C) to node []{$\Psi^p_k$} (D);
  \draw[->] (P) to node [] {$\Psi^d_k$} (Ai);
   \draw[->](Ai) to node []{$\Upsilon^{-1}$}(D);
\end{tikzpicture}
\end{center}
\caption{\footnotesize  The commutation diagram for computation of $P_k=\op{R}_k(P_0)$ via max-plus dual and primal space fundamental solution semigroups.}
\label{fig:dual-primal}
\end{figure}

Comparing \er{eq:dual-fund} and \er{eq:Pk-primal-repre}, the representation of the DRE solution $P_k=\op{R}_k(P_0)$ in terms of the max-plus primal space fundamental solution semigroup $\{\Lambda_k,k\in\N\}$ via \er{eq:Pk-primal-repre} has a simpler form than the representation of \er{eq:dual-fund} via the max-plus dual space fundamental solution semigroup $\{\Theta_k,k\in\N\}$. Here, $\{\Lambda_k,k\in\N\}$ is developed directly in the max-plus primal space, and avoids transformation $\Upsilon$ and $\Upsilon^{-1}$ between the max-plus primal and dual spaces.

\section{Connections between the  max-plus dual space and primal space fundamental solution semigroups}
\label{sec:connection}
A solution $P_k=\op{R}_k(P_0)$ of the DRE \er{eq:DRE} can be represented either by the max-plus dual space fundamental solution semigroup $\{\Theta_k, k\in\N\}$ via \er{eq:dual-fund} or by the primal space fundamental solution semigroup $\{\Lambda_k, k\in\N\}$ via \er{eq:Pk-primal-repre}.  This suggests that there exists a correspondence between elements of these two semigroups. This section explores the relationship between these two semigroups.  

Note that functions $\mathrm{B}_k$,  $\mathrm{S}_k$ of \er{eq:Bk}, \er{eq:kernel-primal} are derived from $\widehat{\mathrm{S}}_k$ of \er{eq:hatSk} via the operator $\op{D}_\psi$ of \er{eq:op-dual}.
Thus, both $\Theta_k$ and $\Lambda_k$, the Hessians of $\mathrm{B}_k$ and $\mathrm{S}_k$, are related to $Q_k$, the Hessian of $\widehat{\mathrm{S}}_k$. In order to find the connection between $\Theta_k$ and $\Lambda_k$, the connections between $\Theta_k$ and $Q_k$, and connections between $\Lambda_k$ and $Q_k$ are established first. To this end, define an operator $\Gamma:\M^{2n\times 2n}\rightarrow\M^{2n\times 2n}$ by
\begin{align}
\label{eq:Gamma}
&\Gamma(\Omega)\doteq\Gamma\left(\left[\ba{cc} \Omega^{11}&\Omega^{12}\\\Omega^{21}&\Omega^{22}\ea\right]\right)
\\\nn
       &= \left[\! \ba{cc}M-M
       (M \! + \! \Omega^{11})^{-1}
       M&M(M\! +\! \Omega^{11})^{-1}\Omega^{12}\\\Omega^{21}(M\! + \! \Omega^{11})^{-1}M&\Omega^{22}-\Omega^{21}(M\! + \!\Omega^{11})^{-1}\Omega^{12} \ea\! \right]
\end{align}
for $\Omega\in\M^{2n\times 2n}$ such that $\Omega^{11}+M<0$. It can be verified directly that the inverse of $\Gamma$ is
 \begin{align}
\nn 
 &\Gamma^{-1}(\Omega)\doteq\Gamma^{-1}\left(\left[\ba{cc} \Omega^{11}&\Omega^{12}\\\Omega^{21}&\Omega^{22}\ea\right]\right)
 \\\nn
 &=\left[\!\ba{cc} M(M \! - \! \Omega^{11})^{-1}M\! - \! M&M(M-\Omega^{11})^{-1}\Omega^{12}\\ \Omega^{21}(M \! - \! \Omega^{11})^{-1}M&\Omega^{21}(M \! - \! \Omega^{11})^{-1}\Omega^{12}+\Omega^{22} \ea\!\right]
 \\
 &=-\Gamma(-\Omega)
 \label{eq:inv-Gamma}
\end{align}
for $\Omega\in\M^{2n\times 2n}$ such that $M-\Omega^{11}<0$. 
It has been shown in Theorem 3.8 of \cite{ZD3:13} that $\Theta_k$ and $Q_k$ are connected by
\begin{align}
\label{eq:map-Q-T}
Q_k=\Gamma(\Theta_k),\quad \Theta_k=\Gamma^{-1}(Q_k),
\end{align}
for any $k\in\N$. Note that $Q_k^{11}=\op{R}_k(M)>M$ and $\Theta_k^{11}=M(M-Q_k^{11})^{-1}M-M<-M$ by \er{eq:dynamics-Q}. Thus, both operations $\Gamma({\Theta_k})$ and $\Gamma^{-1}(Q_k)$ are well defined. To establish the connection between $\Lambda_k$ and $Q_k$, define an operator $\Delta:\M^{2n\times 2n}\rightarrow\M^{2n\times 2n}$ by
\begin{align}
\label{eq:Delta}
\Delta(\Omega)&\doteq\Delta\left(\left[\ba{cc}\Omega_{11} &\Omega_{12}\\ \Omega_{21} & \Omega_{22}\ea\right]\right)\doteq \left[\ba{cc}\Omega_{22} &\Omega_{21}\\ \Omega_{12} & \Omega_{11}\ea\right]
\end{align}
for any $\Omega\in\M^{2n\times 2n}$, and a second operation $\Pi:\M^{2n\times 2n}\rightarrow\M^{2n\times 2n}$ via the composition 
\begin{align}
\label{eq:Pi}
\Pi\doteq \Delta\Gamma\Delta,
\end{align}
 where $\Gamma,\Delta$ are as per \er{eq:Gamma} and \er{eq:Delta}.  By inspection of \er{eq:Delta}, $\Delta^{-1}=\Delta$. Consequently, for any $\Omega\in\M^{2n\times 2n}$ such that $M-\Omega^{22}<0$, 
 \begin{align}
 \label{eq:Pi-inv}
 & \Pi^{-1}(\Omega)=(\Delta\Gamma\Delta)^{-1}(\Omega)=(\Delta^{-1}\Gamma^{-1}\Delta^{-1})(\Omega)
 \\\nn
                           &=\Delta(-\Gamma(-\Delta(\Omega)))=-(\Delta\Gamma\Delta)(-\Omega)
 		=-\Pi(-\Omega).
 \end{align}
The matrix operations $\Pi$ and $\Pi^{-1}$ characterise the connection between $Q_k$ of \er{eq:quadr-hatSk} and $\Lambda_k$ of \er{eq:quad-primal} for any $k\in\N$. 
\begin{theorem}
\label{thm:Qk-Lamk}
The matrices $Q_k$ of \er{eq:quadr-hatSk} and $\Lambda_k$ of \er{eq:quad-primal} satisfy
\begin{align}
\label{eq:thm-Qk-Lamk}
Q_{k}=\Pi(\Lambda_{k}), \quad \Lambda_k=\Pi^{-1}(Q_k)
\end{align}
for any $k\in\N$.
\end{theorem}
\begin{proof}
From \er{eq:dynamics-Q}, \er{eq:mono-Q22}, $Q_k^{22}$ is strictly increasing with respect to $k$ with $Q_0^{22}=M$. Consequently, $Q_k^{22}>M$ for all $k\in\N$. Thus, $\Pi^{-1}(Q_k)$ is well defined for all $k\in\N$. Fix any $x,y\in\R^n$, $k\in\N$. From the definition \er{eq:Pi} of $\Pi$,
\begin{align}
\nn
&{\mathrm{S}}_{k}(x,y)={\frac{1}{2}}\left[\ba{c} x\\y\ea\right]^T\Lambda_k\left[\ba{c} x\\y\ea\right]=\left(\op{D}_\psi\widehat{S}_k(x,\cdot)\right)(y)
\\\nn
&=-\int_{\R^n}^\oplus \psi(z,y)\otimes (-\widehat{\mathrm{S}}_k(x,z))\,dz
\\\nn
                      &=-\max_{z\in\R^n}\left\{\psi(z,y)-\frac{1}{2}\left[\ba{c}x \\ z \ea\right]^T\left[\ba{cc}Q_{k}^{11}&Q_{k}^{12}\\Q_{k}^{21}&Q_{k}^{22}\ea\right] \left[\ba{c}x \\ z \ea\right]  \right\}
\\\nn
                      &=-\max_{z\in\R^n}\left\{\psi(z,y)+\frac{1}{2}\left[\ba{c}z \\ x \ea\right]^T\Delta(-Q_{k}) \left[\ba{c}z \\ x \ea\right]  \right\}
\\\nn
                      &=-\frac{1}{2}\left[\ba{c}y \\ x \ea\right]^T\Gamma(\Delta(-Q_{k}))\left[\ba{c}y \\ x \ea\right]
\\\nn
                     &=\frac{1}{2}\left[\ba{c}x \\ y \ea\right]^T(-\Delta\Gamma\Delta)(-Q_{k})\left[\ba{c}x \\ y \ea\right]
\\\nn
&
=\frac{1}{2}\left[\ba{c}x \\ y \ea\right]^T\Pi^{-1}(Q_{k})\left[\ba{c}x \\ y \ea\right].
\end{align}
That is, $\Lambda_k=\Pi^{-1}(Q_k)$. Since $\Pi$ is an invertible operator, it follows that $Q_k=\Pi(\Lambda_k)$.
\end{proof}

Combining \er{eq:map-Q-T} and \er{eq:thm-Qk-Lamk} yields a correspondence between $\Lambda_k$ and $\Theta_k$ for any $k\in\N$. Define a matrix operation $\Xi:\M^{2n\times 2n}\rightarrow\M^{2n\times 2n}$ by
\begin{align}
\label{eq:Xi}
\Xi\doteq\Pi^{-1}\Gamma.
\end{align} 
The inverse $\Xi^{-1}$ is given by 
$$
\Xi^{-1}(\Omega)=(\Pi^{-1}\Gamma)^{-1}(\Omega)=(\Gamma^{-1}\Pi)(\Omega)=-\Gamma(-\Pi(\Omega))
$$
for all $\Omega\in\M^{2n\times 2n}$ such that $\Omega^{22}+M<0$. 
\begin{theorem}
Elements $\Lambda_k$ and $\Theta_k$ of the max-plus primal and dual space fundamental solution semigroups $\{\Lambda_k, k\in\N\}$ and  $\{\Theta_k, k\in\N\}$ (respectively) satisfy
\begin{align}
\label{eq:Thek-Lamk}
\Lambda_k=\Xi(\Theta_k)
\end{align}
for any $k\in\N$.
\end{theorem}
\begin{proof}
Fix any $k\in\N$. From  \er{eq:map-Q-T} and \er{eq:thm-Qk-Lamk},
\begin{align}
\nn
\Lambda_k&=\Pi^{-1}(Q_k)=\Pi^{-1}(\Gamma(\Theta_k))
=(\Pi^{-1}\Gamma)(\Theta_k)=\Xi(\Theta_k).
\end{align}
\end{proof}

Connections among matrices $\Lambda_k, Q_k, \Theta_k$ follow \er{eq:map-Q-T}, \er{eq:thm-Qk-Lamk} and \er{eq:Thek-Lamk} and are shown by the commutation diagram Figure \ref{fig:correspondence}.
\begin{figure}[h]
\begin{center}
\tikzset{node distance=3cm, auto}
\begin{tikzpicture}
  \node (C) {$\Lambda_k$};
  \node (P) [below of=C] {$Q_k$};
  \node (Ai) [right of=P] {$\Theta_k$};
 
  \draw[->] (Ai.140) to node {$\Xi$} (C.-40);
  \draw[->] (C.-20) to node []{$\Xi^{-1}$} (Ai.120);
 \draw[->] (C.-100) to node [swap] {$\Pi$} (P.100);
  \draw[->] (P.80) to node [swap] {$\Pi^{-1}$} (C.-80);
   \draw[->] (P.20) to node [] {$\Gamma^{-1}$} (Ai.160);
  \draw[->] (Ai.180) to node [] {$\Gamma$} (P.00);
\end{tikzpicture}
\end{center}
\caption{\footnotesize Commutation diagram describes connections between matrices $ Q_k, \Theta_k$, and $\Lambda_k$ of  \er{eq:quadr-hatSk}, \er{eq:Thetak}, and \er{eq:quad-primal}. }
\label{fig:correspondence}
\end{figure}
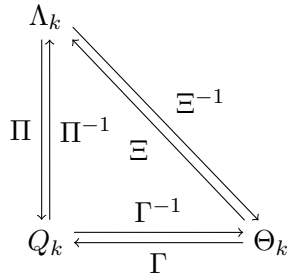

\if{false}

\section{An example}
\label{sec:example}

An example is presented to demonstrate the solution of DRE \er{eq:DRE} via the max-plus dual space fundamental solution semigroup  \er{eq:dual-fund}, and the max-plus primal space fundamental solution \er{eq:Pk-primal-repre}. Consider a DRE of  \er{eq:DRE} specified by $\gamma\doteq \sqrt{6}$, and
\begin{align}
\label{eq:ex}
A&\doteq\left[\ba{cc}0.2&-1.1\\ -0.5&0.7\ea\right], B\doteq\left[\ba{cc} 0.2&0\\0&-0.4\ea\right], 
\\\nn
\Phi&\doteq\left[\ba{cc} 1.2&0\\0&0.8 \ea\right],\quad M\doteq\left[\ba{cc} -10 &0\\0&-10 \ea\right].
\end{align}
The aim is to compute the solution $P_{32}=\op{R}_{32}(P_0)$ with 
$
P_0\doteq\left[\ba{cc}-1&0.4\\0.4&-4 \ea\right]. 
$
Directly applying the operator $\op{R}$ of \er{eq:Riccati-op} iteratively $32$ times yields the solution 
\begin{align}
\label{eq:solution}
P_{32}=\op{R}_{32}(P_0)=\left[\ba{rr}-2.4500&7.5576\\7.5576&-14.1006 \ea\right]. 
\end{align}

Assumption \ref{ass:M} may be verified to hold for the problem data defined by \er{eq:ex}. Hence, formulae  \er{eq:dual-fund}  and \er{eq:Pk-primal-repre} may be applied to solve DRE \er{eq:DRE}. To this end, note that the initial Hessians $\Theta_1$ and $\Lambda_1$ may be computed from \er{eq:Theta1} and \er{eq:Lambda1}, yielding
$$
\Theta_1=\left[\ba{rr|rr}
2.9225  & -7.7408  & -6.6557  &  1.4841
\\
   -7.7408 &  24.1983 &  16.0934 & -10.9020
   \\\hline
   -6.6557 &  16.0934  &  8.4693 & -11.5210
   \\
    1.4841 & -10.9020 & -11.5210 &  -1.2841
\ea\right]
$$
and
$$	
\Lambda_{1}
	= \left[\ba{rr|rr}
	-14.1750  & 46.1250 &  30.0000 & -18.7500
	\\
   46.1250 & -199.0750 & -165.0000 &  26.2500
   \\\hline
   30.0000 & -165.0000 & -150.0000   &      0.0000
   \\
  -18.7500 &  26.2500 &        0.0000 &  -37.5000
	\ea\right].
$$
Then, the matrices $\Theta_{32}$ and $\Lambda_{32}$ can be computed using  \er{eq:propaga-Theta} and \er{eq:circledast-Lamk}, respectively, yielding
\begin{align}
	\Theta_{32}= \left[\ba{rr|rr}		
    5.3171 &  -8.5259 &  -0.0899 &   0.0786
    \\
   -8.5259 &  18.4613 &   0.1817 &  -0.1588
   \\\hline\label{eq:Theta32}
   -0.0899  &  0.1817 &  -2.0896 &  -5.7750
\\
    0.0786 &  -0.1588 &  -5.7750 &  -3.0515
\ea\right]
\end{align}
and
\begin{align}
	\Lambda_{32}
	= \left[\ba{rr|rr}
	 -2.4670  &  7.5925  &  0.0618 &  -0.0370
	 \\
    7.5925 & -14.1722 &  -0.1267  &  0.0759
    \\\hline\label{eq:Lambda32}
    0.0618 &  -0.1267 & -22.1491 & -26.7161
    \\
   -0.0370  &  0.0759 & -26.7161 & -26.5980
	\ea\right].
\end{align}
With a view to apply formula \er{eq:dual-fund}, note that
$$
O_0=\Upsilon(P_0)=\left[\ba{rr}-1.1441 &   0.7429\\
    0.7429 &  -6.7162 \ea\right]. 
$$
With the matrices $\Theta_{32}$ and $\Lambda_{32}$ computed as per \er{eq:Theta32} and \er{eq:Lambda32}, respectively. A straightforward application of formulas \er{eq:dual-fund} and \er{eq:Pk-primal-repre} yields the solution \er{eq:solution}.   

Since formula \er{eq:Pk-primal-repre} employs the max-plus primal space fundamental solution, it does not require the two steps using operators $\Gamma$ and $\Gamma^{-1}$ in formula \er{eq:dual-fund}. In this specific example, this corresponds to using approximately $25\%$ less computational time in obtaining the same solution $P_{32}$ of \er{eq:solution} than using formula  \er{eq:dual-fund}.

\fi

\section{Conclusions}
\label{sec:con}

A new max-plus fundamental solution semigroup is developed for a class of difference Riccati equations (DREs). This max-plus fundamental solution semigroup admits computation of all solutions of a DRE in a specified class, without invocation of duality via the Legendre-Fenchel transform. Connections between this new primal space fundamental solution semigroup, and the previously known dual space fundamental solution semigroup, are established.


\bibliographystyle{plain}
\bibliography{L2}

\begin{thebibliography}{10}

\bibitem{AM:89}
B.D.O. Anderson and J.B. Moore.
\newblock {\em Optimal Control: Linear Quadratic Methods}.
\newblock Prentice-Hall, Englewood Cliffs, N.J., 1989.

\bibitem{B:05}
D.P. Bertsekas.
\newblock {\em Dynamic Programming and Optimal Control}.
\newblock Athena Scientific, 3rd edition, 2005.

\bibitem{BV:04}
S.~Boyd and L.~Vandenberghe.
\newblock {\em Convex Optimisation}.
\newblock Cambridge University Press, 2004.

\bibitem{DM:73}
E.J. Davison and M.C. Maki.
\newblock The numerical solution of the matrix {R}iccati differential equation.
\newblock {\em IEEE Transactions on Automatic Control}, 18:71--73, 1973.

\bibitem{DM1:11}
P.M. Dower and W.M. McEneaney.
\newblock A max-plus based fundamental solution for a class of infinite
  dimensional {R}iccati equations.
\newblock In {\em Proc. $50^{\textstyle{th}}$ IEEE Conference on Decision and
  Control / European Control Conference}, pages 615--620, 2011.

\bibitem{DM2:12}
P.M. Dower and W.M. McEneaney.
\newblock A max-plus method for optimal control of a diffusion equation.
\newblock In {\em Proc. $51^{st}$ IEEE Conference on Decision \& Control (Maui
  HI, USA)}, pages 618--623, 2012.

\bibitem{DM1:12}
P.M. Dower and W.M. McEneaney.
\newblock A max-plus dual space fundamental solution for a class of operator
  differential {R}iccati equations.
\newblock {\em In review, SIAM J. Control \& Optimization, preprint
  arXiv:1404.7209}, 2014.

\bibitem{KTKR:06}
C~Kojima, K.~Takaba, O.~Kaneko, and P.~Rapisarda.
\newblock A characterization of solutions of the discrete-time algebraic
  {R}iccati equation based on quadratic difference forms.
\newblock {\em Linear Algebra and its Applications}, 416:1060--1082, 2006.

\bibitem{M:06}
W.M. McEneaney.
\newblock {\em Max-plus Methods for Nonlinear Control and Estimation}.
\newblock Systems \& Control: Foundations \& Applications. Birkhauser, 2006.

\bibitem{M:08}
W.M. McEneaney.
\newblock A new fundamental solution for differential {R}iccati equations
  arising in control.
\newblock {\em Automatica}, 44(4):920--936, 2008.

\bibitem{SS:98}
A.A. Stoorvogel and A.~Saberi.
\newblock The discrete algebraic {R}iccati equation and linear matrix
  inequality.
\newblock {\em Linear Algebra and its Applications}, 274:317--365, 1998.

\bibitem{ZD3:13}
H.~Zhang and P.M. Dower.
\newblock A max-plus based fundamental solution for a class of discrete time
  linear regulator problems.
\newblock {\em In review for Linear Algebra and its Applications, preprint
  arXiv:1306.5060}, 2013.

\bibitem{ZD4:13}
H.~Zhang and P.M. Dower.
\newblock Analysis of difference {R}iccati equations via a new max-plus based
  fundamental solution.
\newblock In {\em Proceedings of MTNS2014}, pages 679--685, 2014.

\bibitem{ZDG:96}
K.~Zhou, J.~Doyle, and K.~Glover.
\newblock {\em Robust and Optimal Control}.
\newblock Prentice-Hall, Upper Saddle River, N.J., 1996.

\end{thebibliography}

\end{document}